\documentclass[12pt]{amsart}  
\usepackage{amsmath,amsthm,amssymb,amsfonts,eucal,fullpage,latexsym,mathrsfs,stmaryrd}
\usepackage{enumerate}
\usepackage[dvipsnames]{xcolor}
\usepackage{lipsum}
\usepackage{array}
\usepackage{lineno}
\theoremstyle{plain}
\numberwithin{equation}{section}
\theoremstyle{plain}
\newtheorem{Proposition}[equation]{Proposition}

\newtheorem*{Corollary*}{Corollary}
\newtheorem{Theorem}[equation]{Theorem}
\newtheorem*{Theorem*}{Theorem}
\newtheorem{Lemma}[equation]{Lemma}
\theoremstyle{definition}
\newtheorem{Definition}[equation]{Definition}

\newtheorem{Remark}[equation]{Remark}

\def\phi{\varphi}

\renewcommand{\leq}{\leqslant}
\renewcommand{\geq}{\geqslant}

\newcommand{\D}{\mathbb{D}}

\newcommand{\Pol}{\mathcal{P}}

\begin{document}



\title{Optimal Polynomial Approximants in $H^p$}
\author[Centner]{Raymond Centner}
\address{Department of Mathematics and Statistics, University of Southern Maine, Portland, ME 04104, USA. } \email{raymond.centner@maine.edu}
\author[Cheng]{Raymond Cheng}
\address{Department of Mathematics and Statistics, Old Dominion University, Norfolk, VA 23529, USA. } \email{rcheng@odu.edu}
\author[Felder]{Christopher Felder}
\address{Department of Mathematics, Indiana University, Bloomington, IN 47405, USA. } \email{cfelder@iu.edu}
\date{\today}

\begin{abstract} 
This work studies optimal polynomial approximants (OPAs) in the classical Hardy spaces on the unit disk, $H^p$ ($1 < p < \infty$).
In particular, we uncover some estimates concerning the OPAs of degree zero and one. It is also shown that if $f \in H^p$ is an inner function, or if $p>2$ is an even integer, then the roots of the nontrivial OPA for $1/f$ are bounded from the origin by a distance depending only on $p$.  For $p\neq 2$, these results are made possible by the novel use of a family of inequalities which are derived from a Banach space analogue of the Pythagorean theorem.

\end{abstract}

\maketitle


\tableofcontents


\section{Introduction}

For $1 < p < \infty$, we let $H^p$ denote the classical Hardy space of holomorphic functions on the unit disk $\D$, 
\[
H^p := \left\{f \in \operatorname{Hol}(\D) : \sup_{0\le r < 1} \frac{1}{2\pi}\int_0^{2\pi} |f(re^{i\theta})|^p \ d\theta < \infty \right\}.
\]
As is standard, for $f\in H^p$, we adopt the norm notation 
\[
\|f\|_p := \left(\sup_{0\le r < 1} \frac{1}{2\pi}\int_0^{2\pi} |f(re^{i\theta})|^p \ d\theta \right)^{1/p}.
\]
Given a non-negative integer $n$, the \textit{$n$-th optimal polynomial approximant} to $1/f$ is the polynomial solving the minimization problem
\[
\min_{q \in \Pol_n}\|1 - qf\|_p,
\]
where $\Pol_n$ is the set of complex polynomials of degree less than or equal to $n$. 
For brevity, we shall use the acronym \textit{OPA} to mean optimal polynomial approximant. 
Notationally, for fixed $n$ and $f \in H^p$, we denote this polynomial by $q_{n,p}[f]$.
It is important to note that when $1 < p < \infty$, the space $H^p$ is \textit{uniformly convex}. Consequently, $H^p$ possesses the \textit{unique nearest point property}: if $\mathcal{V} \subseteq H^p$ is a closed linear subspace and $h \in H^p$, then there is a unique vector $u \in \mathcal{V}$ such that $u$ solves the minimization problem
\[
\min_{\mu \in \mathcal{V}} \| \mu - h \|_p.
\]
The function $u$ is referred to as the \textit{metric projection} of $h$ onto $\mathcal{V}$. When $p=2$, this projection is the usual orthogonal projection; when $p\neq 2$, the metric projection is generally nonlinear. Consequently, in the latter case, the metric projection does not come equipped with the explicit computational benefits enjoyed in the Hilbert space case.

In terms of OPAs, we see that $q_{n,p}[f]$ is simply the metric projection of the unit constant function onto the subspace of $H^p$ spanned by $\{f, zf, z^2 f,\ldots,$  $z^n f\}$.  Heuristically, we see that if a polynomial $Q$ looks like $1/f$, then $Qf - 1$ looks like zero; thus, $q_{n,p}[f]$ is indeed an approximation to $1/f$. As $n$ is exhausted ad infinitum, $\|q_{n,p}[f]f-1\|_p$ decreases \textit{optimally}.

The genesis of OPAs is in the setting of $H^2$ and first appeared in engineering literature with respect to the study of digital filters in signal processing. More recently, OPAs have attracted interest among mathematicians working in various settings (see, for example, the surveys by B\'{e}n\'{e}teau and the first named author \cite{BC}, and Seco \cite{Sec}). This revival has been primarily informed by the hunt for the classification of cyclic functions. In a space of functions where the forward shift is bounded and the polynomials are dense, it is elementary to see that the linear span of orbit of the function 1 under the shift is dense in the space itself. A function with this property is referred to as \textit{cyclic}. In $H^2$, it is well known that such vectors have an explicit function theoretic form, known as \textit{outer}. However, there is  little known outside of $H^2$.

Conceptually, OPAs extend directly to any normed function space that is uniformly convex.
They have been explored in Dirichlet type spaces, 
\cite{BCLSS1,BFKSS,BKLSS};
more general reproducing kernel Hilbert spaces  
\cite{FMS};
the spaces $\ell^p_A$ \cite{CRSX,ST};
$L^p$ of the unit circle  \cite{Cent};
Hilbert spaces of analytic 
functions on the bidisk  
\cite{BCLSS2,BKKLSS} and the unit ball
\cite{SS2,SS1}.
They have been studied in the context of free polynomials as well
\cite{AAJS}. 

A characterization for cyclic functions in $H^p$ is well-known (as a consequence of the classical Beurling Theorem), so the motivation of cyclicity does not directly carry over to the exploration here. However, the Banach space geometry makes the study of these types of approximation problems much more challenging; the analysis becomes nonlinear when $p\neq 2$. Further, as we shall see, these problems require novel techniques.

The objective of this paper is to more closely study some of the characteristics of OPAs in $H^p$. This study originated in \cite{Cent}, where OPAs in $L^p$ (of the unit circle) and $H^p$ were first considered.  Since OPAs were originally created to be implemented in the design of digital filters, it was natural to consider how a change to the $H^p$ setting would affect the properties of the filter.  We continue to explore OPAs in $H^p$ here, $1<p<\infty$. We note that this exploration is distinctly different from the $p=2$ case; in that setting, much is known--  thanks in part to the Hilbert space geometry. In $H^2$, the orthogonal projection provides a direct connection to many developed areas of analysis, including the theory of reproducing kernels and orthogonal polynomials. We point to \cite{MR3614926} for results in this arena, including lower bounds on the moduli of the root of an OPA (there, it is shown that an OPA cannot vanish inside the closed unit disk). 

The outline of the paper is as follows: 
in Section \ref{prelim}, we provide some background and notation, introducing extensions of the Pythagorean inequality, which will later be used in lieu of linearity when $p\neq2$. 
Sections \ref{deg_0} and \ref{deg_1} establish results concerning the coefficients of the constant and linear OPAs associated to a function in $H^p$, respectively. 
We end this work in Section \ref{roots}, with results about the location of OPA roots in two cases---when $f$ is an inner function and when $p>2$ an even integer.


\section{Notation, Orthogonality, and Pythagorean Inequalities}\label{prelim}

In this section we establish some notation, review a notion of orthogonality for normed spaces, and present an associated version of the Pythagorean theorem that applies to the $L^p$ spaces. This extension of the Pythagorean theorem takes the form of a family of inequalities.  They will play a central role in our study of OPAs.

Let $\mathbf{x}$ and $\mathbf{y}$ be vectors belonging to a normed linear space $\mathscr{X}$.  We say that $\mathbf{x}$ is {\em orthogonal} to $\mathbf{y}$ in the Birkhoff-James sense \cite{AMW, Jam}  if
 \[    \|  \mathbf{x} + \beta \mathbf{y} \|_{\mathscr{X}} \geq \|\mathbf{x}\|_{\mathscr{X}}  \]
for all scalars $\beta$. 
In this situation we write $\mathbf{x} \perp_{\mathscr{X}} \mathbf{y}$.

There are other ways to generalize orthogonality from an inner product space to a normed space, but this approach is particularly useful since it is based on an extremal condition.

If $\mathscr{X}$ is an inner product space, then the relation $ \perp_{\mathscr{X}} $ coincides with the usual orthogonality relation.  In more general spaces, however, the relation $\perp_{\mathscr{X}}$ is neither symmetric nor linear.  In case $\mathscr{X} = L^p(\mu)$, for any measure $\mu$,  let us write $\perp_p$ instead of  $\perp_{L^p}$, and similarly for $\mathscr{X} = H^p$.  
 
There is a practical criterion for  $\perp_p$ when $1<p<\infty$.

\begin{Theorem}[James \cite{Jam}]\label{praccritperp}
Suppose that $1<p<\infty$.  Then for $f$ and $g$ belonging to $L^p$ we have
\begin{equation}\label{BJp}
 {f} \perp_{p} {g}  \iff   \int |f|^{p - 2} \overline{f} g \,d\mu  = 0,
\end{equation}
where any occurrence of ``$|0|^{p - 2} \bar{0}$'' in the integrand is interpreted as zero.
\end{Theorem}

In light of \eqref{BJp} we define, for a complex number $\alpha=re^{i\theta}$, and any $s > 0$, the quantity
\[
\alpha^{\langle s \rangle} = (re^{i\theta})^{\langle s \rangle} := r^{s} e^{-i\theta}.
\]
It is straightforward to verify that for any complex numbers $\alpha$ and $\beta$,  exponent $s>0$, and integer $n \geq 0$, we have
\begin{align*}
(\alpha\beta)^{\langle s\rangle} &= \alpha^{\langle s\rangle} \beta^{\langle s\rangle}\\
|\alpha^{\langle s \rangle}| &= |\alpha|^s\\
\alpha^{\langle s \rangle}\alpha &= |\alpha|^{s+1}\\
(\alpha^{\langle s \rangle})^n &= (\alpha^n)^{\langle s \rangle}\\
(\alpha^{\langle p-1 \rangle})^{\langle p'-1 \rangle} &= \alpha.
\end{align*}

By comparison with the case  $p=2$, we can think of taking the  ${\langle s \rangle}$ power as generalizing complex conjugation.

Throughout this paper, the parameter $p$ satisfies $1<p<\infty$, and we use $p'$ to denote the H\"older conjugate of $p$, that is,
\[
     \frac{1}{p}+\frac{1}{p'} = 1.
\]
If $f \in L^p$, $1<p<\infty$, then $f^{\langle p - 1\rangle} \in L^{p'}$.
Thus, from  \eqref{BJp} we get
\[
f \perp_{p} g \iff \langle g, f^{\langle p - 1\rangle} \rangle = 0.
\]
Consequently, the relation $\perp_{p}$ is linear in its second argument, and it then makes sense to speak of a vector being orthogonal to a subspace of $L^p$.  In particular, if $f \perp_p g$ for all $g$ belonging to a subspace $\mathscr{M}$ of $L^p$, then
\[
       \| f + g \|_p  \geq  \|f\|_p
\]
for all $g \in \mathscr{M}$.  That is, $f$ relates to a nearest-point problem with respect to the subspace $\mathscr{M}$.

Direct calculation will also confirm that
\[
     \langle f, f^{\langle p-1\rangle} \rangle = \|f\|^p_p,
\]
and hence $f^{\langle p-1\rangle}/ \|f\|^{p-1}_p \in L^{p'}$ is the norming functional of $f \in L^p \setminus \{0\}$ (smoothness ensures that the norming functional is unique).

In connection with Birkhoff-James Orthogonality, there is a version of the Pythagorean theorem for $L^p$.
 It takes the form of a family of inequalities relating the lengths of orthogonal vectors with that of their sum.

\begin{Theorem}\label{pythagthm}
Suppose that $x \perp_p y$ in $L^p$.
If $p \in (1, 2]$, then
\begin{align}
   \| x + y \|^p_p & \leq  \|x\|^p_p + \frac{1}{2^{p-1}-1}\|y\|^p_p \label{upper1}\\
    \| x + y \|^2_p & \geq  \|x\|^2_p + (p-1)\|y\|^2_p.\label{lower1}
\end{align}
If $p \in [2, \infty)$, then
\begin{align}
   \| x + y \|^p_p & \geq  \|x\|^p_p + \frac{1}{2^{p-1}-1}\|y\|^p_p \label{lower2}\\
    \| x + y \|^2_p & \leq  \|x\|^2_p + (p-1)\|y\|^2_p.\label{upper2}
\end{align}
\end{Theorem}

These inequalities originate from \cite{Byn,BD,CMP1}; see \cite[Corollary 3.4]{CR} for a unified treatment with broader classes of spaces.

It will sometimes be expedient to refer to \eqref{upper1} and \eqref{upper2} as the upper Pythagorean inequalities, and \eqref{lower1} and \eqref{lower2} as the lower Pythagorean inequalities, so that the two cases depending on $p$ can be handled together.  The specific values of the positive multiplicative constants, i.e., $p-1$ and $1/(2^{p-1}-1)$, are generally unimportant, and thus they will usually be denoted simply by $K$, possibly with a subscript.

When $p=2$, these Pythagorean inequalities reduce to the familiar Pythagorean theorem for the Hilbert space $L^2$.  More generally, these Pythagorean inequalities enable the application of some Hilbert space methods and techniques to smooth Banach spaces satisfying the weak parallelogram laws; see, for example, \cite[Proposition 4.8.1 and Proposition 4.8.3; Theorem 8.8.1]{CMR}.

Before moving onto results, we recall the notation for the $n$-th optimal polynomial approximant. 
\begin{Definition}[OPA]
Given a non-negative integer $n$ and a function $f \in H^p$, the \textit{$n$-th optimal polynomial apporoximant} to $1/f$ in $H^p$ is the polynomial solving the minimization problem
\[
\min_{q \in \Pol_n}\|1 - qf\|_p,
\]
where $\Pol_n$ is the set of complex polynomials of degree less than or equal to $n$. This polynomial exists, is unique, will be denoted by
\[
q_{n,p}[f].
\]
\end{Definition}

\section{The Degree Zero OPA}\label{deg_0}

In this section, we use the Pythagorean inequalities to obtain upper and lower bounds for $q_{0,p}[f]$, the constant OPA. When $p = 2$, we can explicitly calculate OPAs with routine Hilbert space (linear algebra) methods. When $p \neq 2$, the Pythagorean formula manifests itself in the form of inequalities. In turn, we trade direct calculation for estimates. We use this updated Pythagorean toolbox to culminate the results in the present section. 

From \cite[Proposition 5.2]{Cent}, we know that $f(0)=0$ if and only if $q_{0,p}[f]$ vanishes identically.  Thus, to avoid the trivial case, we assume that $f(0) \neq 0$.

Fix $1<p<\infty$.    For $f \in H^p$, let $\lambda\in\mathbb{C}$ be the constant OPA $q_{0,p}[f]$, that is,
\[
       \| \lambda f - 1 \|_p   =  \inf_{c\in\mathbb{C}} \|cf - 1\|_p.
\]

Let us adopt the following notation for the Pythagorean inequalities:  when $f \perp_p g$,
\begin{align}
    \|f + g\|_p^r &\geq \|f\|_p^r + K_1 \|g\|_p^r  \label{pythag1}\\
    \|f + g\|_p^s &\leq \|f\|_p^s + K_2 \|g\|_p^s. \label{pythag2}
\end{align}

\textbf{Remark}: If $g$ is small compared to $f$, then \eqref{pythag2} is sharper than the triangle inequality.  

With that, here is an upper bound for  $|q_{0,p}[f]|$, in terms of $f$.

\begin{Proposition}  Let $1<p<\infty$, and suppose that $f \in H^p$ satisfies $f(0)\neq 0$.  If $\lambda = q_{0,p}[f]$, then
    \[
     |\lambda| \leq \frac{1}{K_1^{1/r} \big[ \|f - f(0)\|_p^r + \|f\|_p^r   \big]^{1/r}},
\]
where $r$ and $K_1$ are the lower Pythagorean parameters.
\end{Proposition}

\begin{proof}
First note that
\[
       1 \perp_p z^k
\]
for all $k \geq 1$, due to $\int |1|^{p-2}\bar{1} e^{ik\theta}\,dm = 0$.  Thus for any $f \in H^p$,  
\[
        \lambda f(0) - 1\ \   \perp_p \ \  \lambda(f - f(0)).
\]

The lower Pythagorean inequality, \eqref{pythag1}, gives
\begin{align}
    \|\lambda f - 1\|_p^r &\geq \| \lambda f(0) - 1\|_p^r + K_1 \| \lambda(f - f(0))\|_p^r \label{pythag1a}
\end{align}

Next, by the extremal property of $\lambda$, we also have
\[
     \lambda f - 1 \ \perp_p \ \lambda f.
\]
This yields
\begin{align}
    \|1\|_p^r &\geq \| \lambda f - 1\|_p^r + K_1 \|\lambda f\|_p^r \label{pythag1ab}
\end{align}

Combining \eqref{pythag1a} and \eqref{pythag1ab} takes us to
\begin{align}
       1  &\geq \|\lambda f - 1\|_p^r + K_1 \|\lambda f\|_p^r \nonumber\\
           &\geq \| \lambda f(0) - 1\|_p^r + K_1 \| \lambda(f - f(0))\|_p^r + K_1 \|\lambda f\|_p^r.\label{starthere}
\end{align}
Drop the first term on the right, and isolate $|\lambda|$, resulting in the following bound:
\[
     |\lambda| \leq \frac{1}{K_1^{1/r} \big[ \|f - f(0)\|_p^r + \|f\|_p^r   \big]^{1/r}}.
\]
\end{proof}

We may extract a sharper bound by estimating the first term of \eqref{starthere}, rather than dropping it.

\begin{Proposition}  Let $1<p<\infty$, and suppose that $f \in H^p$ satisfies $f(0)\neq 0$.  If $\lambda = q_{0,p}[f]$, then
\begin{equation}\label{ublamb2}
     |\lambda|^{r-1} \leq \frac{r|f(0)|}{K_1 \| (f - f(0))\|_p^r + K_1 \| f\|_p^r},
\end{equation}
where $r$ and $K_1$ are the lower Pythagorean parameters.
\end{Proposition}

\begin{proof} Starting from \eqref{starthere}, we have
\begin{align*}
       1  &\geq \| \lambda f(0) - 1\|_p^r + K_1 \| \lambda(f - f(0))\|_p^r + K_1 \|\lambda f\|_p^r   \\
       1-\| \lambda f(0) - 1\|_p^r  &\geq  + K_1 \| \lambda(f - f(0))\|_p^r + K_1 \|\lambda f\|_p^r   \\
      r\cdot 1^{r-1} (1-\| \lambda f(0) - 1\|_p)  &\geq  + K_1 \| \lambda(f - f(0))\|_p^r + K_1 \|\lambda f\|_p^r   \\
      r|\lambda f(0)|  &\geq  + K_1 \| \lambda(f - f(0))\|_p^r + K_1 \|\lambda f\|_p^r,
\end{align*}
in which we have employed the elementary inequality
\[
       \frac{b^r - a^r}{b-a} = rc^{r-1} \leq rb^{r-1}, 
\]
for $a<b$ and some $c \in (a,b)$, arising from the mean value theorem.  We thus obtain
\[
     |\lambda|^{r-1} \leq \frac{r|f(0)|}{K_1 \| (f - f(0))\|_p^r + K_1 \| f\|_p^r}.
\]
\end{proof}

The other Pythagorean inequality enables a lower bound for $|q_{0,p}[f]|$.  

\begin{Proposition}  Let $1<p<\infty$, and suppose that $f \in H^p$ satisfies $f(0)\neq 0$.  If $\lambda = q_{0,p}[f]$, then
     for any constant $c$ we have
     \[
          |\lambda|^s  \geq \frac{1 - \|cf-1\|_p^s}{K_2\big( \|f - f(0)\|_p^s +\|f\|_p^s  \big)},
     \]
     where $s$ and $K_2$ are the upper Pythagorean parameters.
\end{Proposition}

\begin{proof}
The upper Pythagorean inequality \eqref{pythag2} provides that
\begin{align*}
    \|\lambda f - 1\|_p^s &\leq \| \lambda f(0) - 1\|_p^s + K_2 \| \lambda(f - f(0))\|_p^s \\
    \|1\|_p^s &\leq \| \lambda f - 1\|_p^s + K_2 \|\lambda f\|_p^s, 
\end{align*}
   where $s$ and $K_2$ are the upper Pythagorean parameters.

   Combining these yields that
   \[
      1 \leq \| \lambda f(0) - 1\|_p^s + K_2 \| \lambda(f - f(0))\|_p^s  + K_2 \|\lambda f\|_p^s.
   \]
   Apply \eqref{pythag1a} to get
   \[
         1 \leq \| \lambda f - 1\|_p^s + K_2 \| \lambda(f - f(0))\|_p^s  + K_2 \|\lambda f\|_p^s.
   \]
   Replacing $\lambda$ with $c$ makes the first term at least as large, so we obtain
    \[
         1 \leq \| c f - 1\|_p^s + K_2 \| \lambda(f - f(0))\|_p^s  + K_2 \|\lambda f\|_p^s,
   \]
   from which the claim follows.
   \end{proof}

In particular, we could choose $c = 1/f(0)$, and then this says
  \[
          |\lambda|^s  \geq \frac{1-  \|(f - f(0))/f(0)\|_p^s}{K_2\big( \|f - f(0)\|_p^s +\|f\|_p^s  \big)}.
     \]
Evidently this result has value only when $f$ has relatively little mass outside of its constant term.


\section{The Degree One OPA}\label{deg_1}

As seen in \cite{Cent}, it can be quite difficult to calculate OPAs for $p\neq 2$.  However, if one were to discover an efficient way to compute the coefficients, then the $H^p$ setting could have use in the design of digital filters.  The reason for this is because the zeros would have a different configuration, and thus the filter would amplify/attenuate different frequencies compared to the $p=2$ case.  Of course, the utility of the filter would rest on whether or not it is stable (i.e., zero-free in $\overline{\D})$.  Therefore, it seems sensible to start by providing estimates for the zeros.  Again, we note that by using the Pythagorean inequalities, we swap direct calculation for estimates. By doing so, we obtain expressions and bounds for the linear OPA root and leading coefficient.  The reason for considering the linear case is due to the following reduction.

\begin{Proposition}\label{linearsuffprop}
    Let $1<p<\infty$, $w \in \mathbb{D} \setminus \{0\}$, and $n \in \mathbb{N}$. 
    Then $w$ is the root of some OPA in $H^p$ if and only if $w$ is the root of some optimal linear approximant in $H^p$.
\end{Proposition}

\begin{proof}
    Let $Q$ be the OPA of degree up to $n$ for some $1/f$, with $f \in H^p$ and $Q(w) = 0$.  Then $Q(z) = (a+bz)Q_0(z)$ for some $Q_0 \in \mathscr{P}_{n-1}$, where $-a/b = w$.   Then
    \begin{align*}
         Qf-1\ \ &\perp_{H^p}  \ \ \mathscr{P}_n f\ \ \ \implies\\
         (a + bz)Q_0(z)f(z) - 1 \ \ &\perp_{H^p} \ \ \mathscr{P}_1 Q_0 f\ \ \ \implies\\
    \end{align*}
    $a + bz$ is the optimal linear approximant to $1/(Q_0f)$.   
    The converse is trivial.
\end{proof}

If we are looking for bounds for the roots of OPAs, we may therefore restrict our attention to the linear ones.

Suppose $1<p<\infty$, and $w \in \mathbb{D}$. 
Let $D_w$ be the difference-quotient operator
\[
        (D_w f)(z) := \frac{f(z) - f(w)}{z-w},
\]
for all $f \in H^p$.  

Recall that 
\[
       D_w = B(I-wB)^{-1},
\]
and consequently
\[
     \|D_w\| \leq \frac{1}{1 - |w|}.
\]

Here is a bound for the root of a linear polynomial, based on boundedness of the difference-quotient.

\begin{Proposition}\label{fromdiffquot}
     Let $1<p<\infty$.  Let $f \in H^p$, and suppose 
     \[
          Q(z)  = Q_0\Big(1 - \frac{z}{z_0}\Big)
     \]
      for some $z_0 \in \mathbb{D}\setminus \{0\}$ and $Q_0 \neq 0$.  Then
     \[
        |z_0|  \geq  \frac{\|Q_0 f\|_p}{\|Q_0 f\|_p + \|Qf - 1\|_p}.
     \]
\end{Proposition}

\begin{proof}
\begin{align*}
        \big(D_{z_0} (Qf-1)\big)(z)  &=   \frac{[Q(z)f(z) - 1]-[Q(z_0)f(z_0) - 1]}{z-z_0}  \\
             &=  \frac{Q(z)f(z)}{z-z_0} \\
             &=   \frac{Q_0 (z_0-z) f(z)}{z_0(z-z_0)} \\
             &=  -\frac{Q_0 f(z)}{z_0}.
\end{align*}

Consequently,
\begin{align*}
     \|\big(D_{z_0} (Qf-1)\big)(z) \|_p  &\leq \|D_{z_0}\| \|Qf-1\|_p \\
        \Big\|  \frac{Q_0 f(z)}{z_0}\Big\|_p   &\leq  \frac{1}{1 - |z_0|}  \|Qf-1\|_p, 
\end{align*}
from which the claim follows.
\end{proof}

Note that the statement and proof of Proposition \ref{fromdiffquot} make no use of the extremal nature of $q_{1,p}[f]$, and hence the conclusion holds for any linear polynomial $Q$.

Next we apply the Pythagorean inequalities to derive upper and lower bounds for the OPA root.

\begin{Proposition}
   Let $1<p<\infty$, and suppose that $f \in H^p$ with $f(0) \neq 0$.  If $Q(z) = a(z-w) = q_{1,p}[f]$, then
   \begin{align*}
        \Bigg(\frac{ \|a zf-1\|_p^2 -\|Qf-1\|_p^2 }{ K_1\|a f\|_p^2 }\Bigg)^{1/2}  &\geq |w| \geq  \Bigg(\frac{ \|a zf-1\|_p^p -\|Qf-1\|_p^p }{ K_2 \|a f\|_p^p}\Bigg)^{1/p},\ \,1<p<2\\
       \Bigg(\frac{ \|a zf-1\|_p^2 -\|Qf-1\|_p^2 }{ K_2\|a f\|_p^2 }\Bigg)^{1/2}  &\leq |w| \leq  \Bigg(\frac{ \|a zf-1\|_p^p -\|Qf-1\|_p^p }{ K_1 \|a f\|_p^p}\Bigg)^{1/p},\ \,2<p<\infty,
   \end{align*}
   where $K_1$ and $K_2$ are respectively the lower and upper Pythagorean constants.
\end{Proposition}

\begin{proof}
If $2<p<\infty$,  then the Pythagorean inequalities are
\begin{align*}
   \| x + y \|^p_p & \geq  \|x\|^p_p + K_1 \|y\|^p_p \\
    \| x + y \|^2_p & \leq  \|x\|^2_p + K_2\|y\|^2_p.
\end{align*}

The condition
\[
         Qf - 1 = a(z - w)f -1  \ \  \perp_p  \ \  f
\]
then enables us to write
\begin{align*}
     a zf -1  &=  a zf - a w f - 1 +a w f \\
     \|a zf-1\|_p^p &\geq \|Qf-1\|^p_p + K_1 \|w a f\|_p^p \\
     |w|^p  &\leq  \frac{ \|a zf-1\|_p^p- \|Qf-1\|^p_p }{ K_1 \|a  f\|_p^p },
\end{align*}
and
\begin{align*}
     \|a zf-1\|_p^2 &\leq \|Qf-1\|^2_p + K_2 \|a w f\|_p^2 \\
     |w|^2  &\geq \frac{  \|a zf-1\|_p^2- \|Qf-1\|^2_p}{ K_2 \|a  f\|_p^2 }.
\end{align*}

Hence we have
\[
       \Bigg(\frac{ \|a zf-1\|_p^2 -\|Qf-1\|_p^2 }{ K_2\|a f\|_p^2 }\Bigg)^{1/2}  \leq |w| \leq  \Bigg(\frac{ \|a zf-1\|_p^p -\|Qf-1\|_p^p }{ K_1 \|a f\|_p^p}\Bigg)^{1/p}.
\]
The steps are similar when $1<p<2$.  
\end{proof}

These bounds are limited by the fact that they depend on the OPA itself.  Nonetheless, we will see that such results have their uses.

Here follows a bound for the coefficient $a$.

\begin{Proposition}\label{aboundfora}
    Let $1<p<\infty$, and suppose that $f \in H^p$ with $f(0)\neq 0$. If $q_{1,p}[f] = a(z-w)$,
    then \[
    |a|^{r-1} \leq \frac{r|wf(0)|}{K_1 \|(z-w)f(z)+wf(0)\|_p^r + K_1\|(z-w)f(z)\|_p^r},
\]
where $r$ and $K_1$ are the lower Pythagorean parameters.
\end{Proposition}

\begin{proof}
Notice that $a = q_{0,p}[(z-w)f]$.  Thus the claim follows from \eqref{ublamb2}, being that
$(z-w)f(z)$ takes the value $-wf(0)$ when $z=0$.
\end{proof}

We close this section with some formulas for the OPA root $w$ and coefficient $a$, made possible by the extremal property of the OPA, and the associated orthogonality relations.  These will be useful when we seek bounds for $w$ in the concluding section.

\begin{Theorem}\label{formulasforw}
Let $1<p<\infty$, and suppose that $f \in H^p$ satisfies $\|f\|_p = 1$ and $f(0)\neq 0$.  Let $Q(z) = q_{1,p}[f](z) = a(z-w)$, and define
\begin{align*}
    A &=  \int |Qf-1|^{p-2}\bar{f}\,dm   \\
    B &=    \int |Qf-1|^{p-2}\overline{zf}\,dm \nonumber  \\
    C &=    \int |Qf-1|^{p-2}|f|^2\,dm  \nonumber   \\
    D &=         \int |Qf-1|^{p-2}\bar{z} |f|^2\,dm, \nonumber 
\end{align*}
provided that the integrals exist.

Then we have
\begin{align}
      w &= \frac{B\overline{D}-AC}{BC - AD} \label{dubform} \\
      w&= \frac{a\bar{D}-A}{aC} \label{dubform2} \\
      w &= \frac{aC-B}{aD} \label{dubform3}\\
      a &= \frac{B-wA}{2C -w\bar{D} -\bar{w}D} \nonumber\\
      a &= \frac{AD-BC}{|D|^2 - C^2}.\nonumber
\end{align} 
\end{Theorem}
   
Note:  These formulas always hold if $2<p<\infty$, and hold for $1<p<2$ as well if $\|Qf\|_{\infty} <1$.  
   
\begin{proof}
    Use the orthogonality properties of $Q$, along with Theorem \ref{praccritperp}.

    Suppose that $q_{1,p}[f] = a(z-w)$.  The associated orthogonality equations tell us that
\begin{align*}
     \int|1 - Q f|^{p-2} (1 - Qf) \overline{f} \,dm &=  0\\
     \int|1 - Q f|^{p-2} (1 - Qf) \overline{zf} \,dm &=  0.
\end{align*} 
Providing that the separate integrals exit, we can rearrange these to get  
\begin{align*}
     \int|1 - Q f|^{p-2} f\,dm &= -\overline{aw} \int|1 - Q f|^{p-2} |f|^2\,dm +\bar{a} \int|1 - Q f|^{p-2} \bar{z} |f|^2\,dm \\
      \int|1 - Q f|^{p-2}z f\,dm &= -\overline{aw} \int|1 - Q f|^{p-2}z |f|^2\,dm +\bar{a} \int|1 - Q f|^{p-2} |f|^2\,dm.
\end{align*}
This linear system can be solved for $a$ and $w$.  The result is
\begin{align}
       \left[\begin{array}{c} -aw \\ a \end{array}\right] &= \left[ \begin{array}{cc} C & \bar{D} \\D & C   \end{array} \right]^{-1}
       \left[\begin{array}{c}  A \\ B \end{array}\right]  \label{recursivescheme} \\
       & \nonumber  \\ 
             &= \frac{1}{|C|^2 + |D|^2}\left[ \begin{array}{rr} C & -\bar{D} \\ -D & C   \end{array} \right]
       \left[\begin{array}{c}  A \\ B \end{array}\right]  \nonumber  \\
       &  \nonumber \\
       &= \frac{1}{|C|^2 + |D|^2}
       \left[\begin{array}{c}  AC-B\bar{D} \\ BC-AD \end{array}\right].  \nonumber 
\end{align}
The formula \eqref{dubform} then derives from 
\[
     w = -\frac{-aw}{a} = - \frac{(AC-B\bar{D})/(|C|^2 + |D|^2)}{(BC-AD)/(|C|^2 + |D|^2)},
\] and the rest is similar.
\end{proof}


\section{OPA Roots Are Bounded From Zero}\label{roots}

When $p=2$, the root of a non-trivial linear OPA must lie outside the unit disk \cite{MR3614926}. When $p\neq 2$, we are able to show that in two special cases the root of a linear OPA must be bounded from the origin.  We suspect that this is true in general, but this seems beyond the reach of our present methods.  We start with the following observation.

\begin{Proposition}
    Let $1<p<\infty$.  Suppose that $f$ is a cyclic function in $H^p$ with $f(0)\neq 0$, and $\{n_k\}_{k\in\mathbb{N}}$ is an increasing sequence of non-negative integers.  If $\{z_{n_k}\}_{k\in\mathbb{N}}$ is a sequence of numbers in $\mathbb{D}$ such that $q_{n_k,p}[f](z_{n_k})=0$, then $|z_{n_k}|\rightarrow 1$ as $k\rightarrow\infty$.  
\end{Proposition}
\begin{proof}
From \cite[Proposition 5.1]{Cent}, we have that 
\begin{equation*}
    \sqrt{1-\|q_{n,k}[f]f-1\|_p^p}\leq |z_{n_k}|\leq 1
\end{equation*}
for each $k\in\mathbb{N}$.  Now, since $f$ is cyclic in $H^p$, it follows that $\|q_{n_k,p}[f]f-1\|_p^p\rightarrow 0$ as $k\rightarrow\infty$.  Therefore, the result follows.
\end{proof}

Next, we list the two main results.

\vspace{.5pc}

\noindent \textbf{Theorem \ref{rootsbddinnerf}.} \textit{    Suppose that $1<p<\infty$.
    Let $f \in H^p$ be an inner function such that $f(0) \neq 0$.  There exists a constant $\rho>0$, depending only on $p$, such that if $q_{1,p}[f](z) = a(z-w)$, then $|w| \geq \rho$.}

\vspace{.5pc}

\vspace{.5pc}

\noindent \textbf{Theorem \ref{rootsbddevenint}.} \textit{Suppose that $p>2$ is an even integer.  Let $f \in H^p$ satisfy $f(0) \neq 0$.  There exists  $\rho>0$, depending only on $p$, such that if    $q_{1,p}[f](z) = a(z-w)$, then $|w| \geq \rho$.}

\vspace{.5pc}

Both of the above theorems will be established via proofs by contradiction.  We will therefore assume, for the sake of argument, that there is a sequence of functions $f_k \in H^p$ such that $\|f_k\|_p = 1$ and $f_k(0)\neq 0$ for each $k$, for which the respective linear OPAs $q_{1,p}[f_k] = a_k(z-w_k)$ satisfy $\lim_{k\rightarrow\infty} w_k = 0$.  We will show, in a set of lemmas, that the corresponding parameters $\lambda_k = q_{0,p}[f_k]$, $a_k$, and $f_k(0)$ must also tend toward zero at certain rates. Those rates will turn out to be incommensurate.  This contradiction will show that the OPA roots $w_k$ cannot be made arbitrarily close to zero.

For ease of reading we will suppress the index $k$ in the above structures, and instead bear in mind throughout that $f(0)$, $\lambda$, $a$, and $w$ are dependent on $f$.
Since the OPA root is independent from rescaling $f$, there is no harm in assuming that $\|f\|_p = 1$. 

In what follows, we will also employ the binomial approximation in the forms
\begin{align*}
      (1 + |x|)^{\alpha}  &\geq  1+(\alpha/2) |x| \\
      (1 - |x|)^{\alpha}  &\leq  1-(\alpha/2) |x| \\
      (1 - |x|)^{\alpha}  &\geq  1-(2\alpha) |x| 
\end{align*}
for $|x|$ sufficiently small.  The factors of 2 are intended to absorb the error, which is on the order of $|\alpha x|^2$.

The first lemma arises from boundedness of the difference-quotient operator on $H^p$.
\begin{Lemma}\label{fromdiffquotient2}
    Let $1<p<\infty$, and let $f \in H^p$ with $f(0) \neq 0$.  If $0<|w|<1$, and  $Q(z)= a(z-w)$, then 
    \[
         \|Qf\|_p \geq (1 - |w|)|a|\|f\|_p.
    \]
\end{Lemma}
\begin{proof}
  Consider again the difference-quotient operator $D_w$ in $H^p$.  
\begin{align}
     D_w (Qf) &=  \frac{a(z-w)f - a(w-w)f(w)}{z-w} \nonumber \\
     &= \frac{a(z-w)f - 0}{z-w} \nonumber \\
          &=  af \nonumber \\
          \|D_w\| \|Qf\|_p  &\geq \|af\|_p \nonumber\\
          \|Qf\|_p  &\geq (1 - |w|)\|af \|_p \nonumber\\
          &= (1 - |w|)|a|\|f\|_p. \label{fromdq22}
\end{align}
\end{proof}

\begin{Lemma}\label{pumpkinpie2}
   Let $1<p<\infty$, $p\neq 2$, and let $f \in H^p$ with $f(0) \neq 0$.  If $Q(z) =q_{1,p}[f](z)= a(z-w)$, then 
        \begin{align*}
        |w|^p &\geq (pK_1/K_2)|a|^{2-p}\ \,\ 1<p<2\\
       |w|^2  &\geq (pK_1/4 K_2)|a|^{p-2},\ \,\ 2<p<\infty
\end{align*}
for $|w|$ sufficiently small.
\end{Lemma}

\begin{proof}
Let us proceed with the case $2<p<\infty$, the other case being similar.
From the relation $Qf-1 \perp_p Qf$, the lower Pythagorean inequality yields
\begin{align}
     1 &=  -(Qf-1) + Qf  \nonumber \\
     1  &\geq \|Qf-1\|_p^p + K_1\|Qf\|_p^p. \label{fromlower22}
\end{align}

Next, from \cite[Proposition~5.1]{Cent}, we know that $|w|^2 \geq 1 - \|Qf-1\|_p^p$.  Combining this with \eqref{fromlower22} results in the bound
\begin{align}
     |w|^2 &\geq 1 - \|Qf-1\|_p^p  \nonumber  \\
     &\geq  1 - (1 - K_1\|Qf\|_p^p)^{2/p}  \nonumber.  
\end{align}
By the hypothesis that $w$ can be made small, this forces $\|Qf\|_p$ to be correspondingly small, so now the binomial approximation allows for
\begin{align*}
   |w|^2  &\geq  \frac{K_1}{p}\|Qf\|_p^p.  \label{ubqf22}
\end{align*}

We extend this further by utilizing Lemma \ref{fromdiffquotient2}, to the effect that
\[
     |w|^2  \geq (K_1/p) \|Qf\|_p^p \geq   (K_1/p)(1 - |w|)^p |a|^p\|f\|_p^p.
\] 
This shows that $a$ is forced to be small as well.

Finally, start with the upper Pythagorean inequality again, then use the Lower Pythagorean inequality, the binomial approximation, and \eqref{fromlower22} to get
\begin{align}
           \|Qf -1\|_p^2 +  K_2 \|awf\|_p^2 &\geq  \|azf+1\|_p^2  \nonumber \\
           |aw|^2  &\geq \frac{\|azf+1\|_p^2  -\|Qf-1\|_p^2}{K_2 \|f\|_p^2}   \nonumber  \\
                &\geq \frac{(1+K_1\|azf\|_p^p)^{2/p}  -1}{K_2 \|f\|_p^2}  \label{stuffing2}\\
                |aw|^2  &\geq \frac{K_1}{pK_2\|f\|_p^2}|a|^p, \label{cranberry2}
\end{align}
for $w$ sufficiently close to zero.  Notice that to reach \eqref{stuffing2}, we utilized the orthogonality relation
\[
        1 \perp_p  azf.
\]
Dividing through \eqref{cranberry2} by $|a|^2$ yields the claim.  
\end{proof}

\bigskip

\begin{Lemma}\label{ginandtonic2}
Suppose $1<p<\infty$, $p\neq 2$, $f \in H^p$, $f(0) \neq 0$, $\|f\|_p = 1$, $\lambda = q_{0,p}[f]$ and $Q(z) = q_{1,p}[f](z) = a(z-w)$.  If $|w|$ can be made arbitrarily small, then
\begin{align*}
       (4K_2/p)|a|^p (1+|w|)^p &\geq K_2 |\lambda|^2,\ \, 1<p<2\\
       pK_2|a|^2(1+|w|)^2 &\geq K_1 |\lambda|^p,\ \,2<p<\infty.
\end{align*}
\end{Lemma}

\begin{proof}
Once again, let us treat the case $2<p<\infty$ in detail.  The upper Pythagorean inequality gives
\begin{align*}
     Qf-1 &=  azf - awf -1\\
     awf +1 &= -(Qf-1) \oplus_p (azf) \\
     \|awf + 1\|_p^2  &\leq \|Qf-1\|_p^2 + K_2 \|azf\|_p^2.
\end{align*}
Note
\begin{align*}
       K_2 |a|^2  &\geq  \|awf + 1\|_p^2 -  \|Qf-1\|_p^2  \\
            &\geq  \|\lambda f - 1\|_p^2 -  \|Qf-1\|_p^2 \\
            &\geq 0.
\end{align*}
If we assume, for the sake of argument, that $|w|$ can be made arbitrarily small, then we have previously established that both $|a|$ and $\|Qf\|_p$ are forced to be small.  In this case $\|Qf-1\|_p$ tends toward unity, and $\|\lambda f-1\|_p$ must do likewise.  The conclusion is that $\lambda$ goes to zero.  

Next,
\begin{align*}
     \lambda f - 1 &\perp_p  \lambda f \\
     1 &= -(\lambda f - 1) \oplus_p \lambda f\\
     1 &\geq \|\lambda f - 1\|_p^p + K_1 |\lambda|^p \|f\|_p^p \\
     1 -\|\lambda f - 1\|_p^p  &\geq  K_1 |\lambda|^p \|f\|_p^p \\
     1 -\|Q f - 1\|_p^p  &\geq  K_1 |\lambda|^p \|f\|_p^p \\
     1 - (1 - K_2\|Qf\|_p^2)^{p/2}  &\geq  K_1 |\lambda|^p \|f\|_p^p \\
     (pK_2)\|Qf\|_p^2  &\geq  K_1 |\lambda|^p \|f\|_p^p \\
      (pK_2)|a|^2(1+|w|)^2\|f\|_p^2  &\geq  K_1 |\lambda|^p \|f\|_p^p \\
     (pK_2)|a|^2(1+|w|)^2 &\geq K_1 |\lambda|^p
\end{align*}
when $|w|$ is small.  

The case $1<p<2$ is handled analogously.
\end{proof}

\bigskip

\begin{Lemma}\label{whitewine2} 
     Suppose $1<p<\infty$, $p\neq 2$, $f \in H^p$, $f(0) \neq 0$, $\|f\|_p = 1$, $\lambda = q_{0,p}[f]$ and $Q(z) = q_{1,p}[f](z) = a(z-w)$.  If $|w|$ can be made arbitrarily small by choice of $f$, then 
     \begin{align*}
           |f(0)| &\leq C|\lambda|^{p-1},\ \,1<p<2\\
           |f(0)|  &\leq  C|\lambda|,\ \,2<p<\infty 
     \end{align*}
     for some positive constant $C$.
\end{Lemma}

\begin{proof}  Yet again let us take $2<p<\infty$.
Assume $c$ is a constant near zero for which $cf(0) > 0$.  Then use
\[
     1 - K_2 \|\lambda f\|_p^2  \leq  \|\lambda f - 1\|_p^2 \leq  \|cf-1\|_p^2 
\]
and
\begin{align}
   \|\lambda f - 1\|_p \leq  \|cf-1\|_p  &=  \|cf(0)-1 + c(f - f(0))\|_p  \nonumber \\
      \|cf-1\|_p^2  &\leq  |cf(0)-1|^2  + K_2 |c|^2 \|f - f(0)\|_p^2  \nonumber \\
      \|cf-1\|_p^2  &\leq  (1 - |cf(0)|)^2  + K_2 |c|^2 \|f - f(0)\|_p^2  \nonumber \\
      \|cf-1\|_p^2  &\leq  (1 - |cf(0)|)  + K_2 |c|^2 \|f - f(0)\|_p^2. \label{buffalowings2}
\end{align}

By basic calculus, the last expression is minimized by taking
\[
     |c| =  \frac{|f(0)|}{2K_2 \|f - f(0)\|_p^2}.
\]

The minimum value attained in the right side of \eqref{buffalowings2} is therefore
\[
   1 - \frac{|f(0)|^{2}}{4K_2\|f-f(0)\|_p^2} < 1.
\]

Consequently we have 
\begin{align*} 1 - 2K_2 \|\lambda f\|_p^2 &\leq 
     1 - K_2 \|\lambda f\|_p^2\\
        &\leq \|cf-1\|_p^2 \\
       &\leq  1 - \frac{|f(0)|^{2}}{4K_2\|f-f(0)\|_p^2},
\end{align*}
which shows that
\[
      |f(0)| \leq  C|\lambda|.
\]

It is substantially similar to prove the claim for $1<p<2$.

\end{proof}

\bigskip

\begin{Lemma}\label{redbellpepper2}
     Suppose $1<p<\infty$, $p\neq 2$, $f \in H^p$, $f(0) \neq 0$, $\|f\|_p = 1$, $\lambda = q_{0,p}[f]$ and $Q(z) = q_{1,p}[f](z) = a(z-w)$.  If $|w|$ can be made arbitrarily small, then
     \begin{align*}
         C_1 |a| &\leq |wf(0)| \leq C_2|a|^{p-1},\ \,1<p<2\\
         C_1|a|^{p-1} &\leq |wf(0)| \leq C_2|a|,\ \, 2<p<\infty 
     \end{align*}
     for some positive constants $C_1$ and $C_2$.
\end{Lemma}

\begin{proof} Again, we treat the case $2<p<\infty$ in detail. The first inequality 
comes from Proposition \ref{aboundfora}, namely,
\[
   |a|^{p-1} \leq \frac{p|wf(0)|}{K_1 \|(z-w)f(z)+wf(0)\|_p^p + K_1\|(z-w)f(z)\|_p^p},
\]
noting that the denominator is well behaved as $|w|$ decreases to zero.

For the second inequality, start with
\begin{align*}
     u(z-v)f -1   &=  uzf -uvf -1 \\
          &= -(uvf(0) +1)  \oplus_p  (uzf - uv(f - f(0)) \\
     \|u(z-v)f -1 \|_p^2 &\leq |uvf(0) +1|^2 + K_2 \|uzf - uv(f - f(0))\|_p^2 \\
      \|Qf -1 \|_p^2 &\leq |uvf(0) +1|^2 + K_2 |u|^2 (1 + 2|v|)^2 \|f\|_p^2.
\end{align*}

We can choose $u$ and $v$ small in magnitude, with $uvf(0) < 0$.  Then the estimates continue
\begin{align}
           \|Qf -1 \|_p^2 &\leq (1 -|uvf(0)|)^2 + K_2 |u|^2 (1 + 2|v|)^2 \|f\|_p^2 \nonumber  \\
                &\leq  (1 -|uvf(0)|) + K_2 |u|^2 (1 + 2|v|)^2 \|f\|_p^2 \nonumber   \\
           1 - K_2\|Qf\|_p^2 &\leq  (1 -|uvf(0)|) + K_2 |u|^2 (1 + 2|v|)^2 \|f\|_p^2  \nonumber  \\
                |uvf(0)| &\leq K_2\|Qf\|_p^2+ K_2 |u|^2 (1 + 2|v|)^2 \|f\|_p^2.  \label{lemon2}
\end{align}
Now use $u = |a|$, $v = -|w|\overline{f(0)}/|f(0)|$ and $\|Qf\|_p \leq |a|(1+|w|)\|f\|_p$.
\end{proof}

\bigskip

Lemmas
\ref{pumpkinpie2}, \ref{ginandtonic2}, \ref{whitewine2}, and \ref{redbellpepper2} above
 show that the quantities $a$, $\lambda$, and $f(0)$ must tend also toward zero at certain rates. We emphasize that these claims are contingent on the working assumption that the OPA root $w$ can be made arbitrarily small by the choice of $f$, a contention we ultimately intend to disprove in Theorems \ref{rootsbddinnerf} and \ref{rootsbddevenint}. Thus, these lemmas have no scope beyond the intended proofs by contradiction.


\begin{Theorem}\label{rootsbddinnerf}
    Suppose that $1<p<\infty$.
    Let $f \in H^p$ be an inner function such that $f(0) \neq 0$.  There exists a constant $\rho>0$, depending only on $p$, such that if $q_{1,p}[f](z) = a(z-w)$, then $|w| \geq \rho$.
\end{Theorem}

\begin{proof}
    The proof will be by contradiction.  Accordingly, we assume that by choice of $f$ satisfying the hypotheses, $w$ can be made arbitrarily close to zero.   Let us write $Q(z) = q_{1,p}[f](z) = a(z-w)$.
    
In the notation of Theorem \ref{formulasforw}, we have the formula
\[
      w = \frac{aC-B}{aD}.
\]

To evaluate this formula for $w$, let us use the binomial series expansion
\[
      (1 - Qf)^{r} = \sum_{k=0}^{\infty} \binom{r}{k} (-Qf)^k,
\]
where $r = (p-2)/2$.  Note that in the resulting integral expressions for $B$, $C$ and $D$ the terms are uniformly bounded, and the series are absolutely convergent if $|a|$ and $|w|$ are sufficiently small. Hence we may interchange the order of integration and summation.

The condition $|z^m f^n| = 1$ a.e.\ holds for all nonnegative $m$ and $n$, and then the estimate
\begin{align*}
       \Big|  \int (Qf)^j(\overline{Qf})^k z^m f^n\,dm \Big| &\leq  \int |(Qf)^j(\overline{Qf})^k z^m f^n|\,dm \\
          &\leq |a|^{j+k}(1 + |w|)^{j+k}
\end{align*}
applies to each of the formulas for $B$, $C$ and $D$.

It follows that
\begin{align}
     |Qf-1|^{p-2} &= 1 - razf +rawf -r\overline{azf} +r\overline{awf} + O(|a|^2) \label{theexpr2}
\end{align}

When integrating against the expression \eqref{theexpr2} to find $B$, $C$, and $D$, respectively, further terms will vanish, leaving
\begin{align*}
    B &=  -ra  + O(|a|^2) \\
    C &=   1 + rawf(0)+ r\overline{awf(0)} + O(|a|^2) \\
    D &=  -r{af(0)} + r{aw\hat{f}_1}  + O(|a|^2), 
\end{align*}
where $\hat{f}_1$ is the 1st Taylor coefficient of $f$. This gives 
\begin{align}
     w &= \frac{aC-B}{aD}  \nonumber\\
      &= \frac{ ra +ra^2wf(0)+ r|a|^2\overline{wf(0)} +a|a|^2+a|aw|^2 + ra + O(|a|^2)}{ -ra^2{f(0)} + ra^2{w\hat{f}_1}+ O(|a|^3)}  \nonumber\\
      &= \frac{ 2r +rawf(0)+ 2r\overline{awf(0)} + O(|a|)}{ -{{ra}f(0)} + {{ra}w\hat{f}_1}+ O(|a|^2)}.\label{finalexpr2}
\end{align}

We previously showed that as $w$ tends toward zero, so do $a$ and $f(0)$, while $\hat{f}_1$ remains bounded.  But the final expression \eqref{finalexpr2} is bounded from zero for $|w|$ sufficiently small, which is a contradiction.
\end{proof}
    
We now turn to the second theorem of this section. 

\begin{Theorem}\label{rootsbddevenint}
      Suppose that $p>2$ is an even integer.  Let $f \in H^p$ satisfy $f(0) \neq 0$.  There exists  $\rho>0$, depending only on $p$, such that if    $q_{1,p}[f](z) = a(z-w)$, then $|w| \geq \rho$.
\end{Theorem}

\begin{Remark} In the following proof, we present the case $p=4$ in detail; after providing a couple of lemmas, we then point out how to extend this argument to the general case when $p>4$ is an even integer.
\end{Remark}

\begin{proof}
As we proceed, there is no harm in assuming that $\|f\|_4 = 1$, since that does not affect the value of $w$.
The proof will be by contradiction.  To that end, suppose that we have a sequence of selections for $f$ satisfying the hypotheses, for which the corresponding values of $w$ tend toward zero.  For visual clarity we suppress the index.  We have previously established that the associated values of $a$ and $f(0)$ also tend toward zero.

Write $Q(z) = q_{1,4}[f](z) = a(z-w)$, and note that
\[
     |Qf-1|^{p-2}  =  |Qf-1|^{4-2} = (1 - Qf)(1-\overline{Qf})  =  1 - Qf - \overline{Qf} + |Qf|^2.
\]   
   
Again, let us use the integral expressions $A$, $B$, $C$ and $D$  and associated formulas from Theorem \ref{formulasforw}.

By taking integrals against $|Qf-1|^{p-2} = 1 - Qf - \overline{Qf} + |Qf|^2$, we obtain
\begin{align*}
  A &= \overline{f(0)} - \int\bar{f}Qf\,dm -\int\overline{f^2Q}\,dm + \int \bar{f}|Qf|^2\,dm\\
  B &= \int \overline{zf}\,dm -\int\overline{zf}Qf\,dm - \int\overline{zf^2Q}\,dm +\int\overline{zf}|Qf|^2\,dm\\
       &=  0 -\int\overline{zf}Qf\,dm - 0  +\int\overline{zf}|Qf|^2\,dm \\
       &=    -a\int |f|^2\,dm + aw\int\bar{z}|f|^2\,dm +\int \overline{zf}|Qf|^2\,dm \\
  C &= \int|f|^2\,dm -\int Qf|f|^2\,dm - \int \overline{Qf}|f|^2\,dm + \int |f|^2|Qf|^2\,dm\\
   D &= \int\bar{z}|f|^2\,dm  -\int\bar{z}Qf|f|^2\,dm  - \int\overline{zQf}|f|^2\,dm +\int\bar{z}|f|^2|Qf|^2\,dm.
\end{align*}
Notice that in the final expression for $B$, one of the integrals was broken out into two terms, namely 
\[
      \int\overline{zf}Qf\,dm = a\int |f|^2\,dm - aw\int\bar{z}|f|^2.
\]

Next, we substitute the resulting expressions into the formula
\[
       w = \frac{aC-B}{aD}
\]
from \eqref{dubform3},
with the result
\begin{align*}
     w &= \Bigg\{ a\int|f|^2\,dm -a\int Qf|f|^2\,dm - a\int \overline{Qf}|f|^2\,dm + a\int |f|^2|Qf|^2\,dm \\
          &\qquad - \bigg[  -a\int |f|^2\,dm + aw\int\bar{z}|f|^2\,dm +\int \overline{zf}|Qf|^2\,dm \bigg] \Bigg\}\\
          &\qquad \div\Bigg\{  a\int\bar{z}|f|^2\,dm  -a\int\bar{z}Qf|f|^2\,dm  - a\int\overline{zQf}|f|^2\,dm +a\int\bar{z}|f|^2|Qf|^2\,dm \Bigg\}\\
          &= \Bigg\{ 2a\int|f|^2\,dm -a\int Qf|f|^2\,dm - a\int \overline{Qf}|f|^2\,dm + a\int |f|^2|Qf|^2\,dm \\
          &\qquad   - aw\int\bar{z}|f|^2\,dm -\int \overline{zf}|Qf|^2\,dm  \Bigg\}\\
          &\qquad \div\Bigg\{  a\int\bar{z}|f|^2\,dm  -a\int\bar{z}Qf|f|^2\,dm  - a\int\overline{zQf}|f|^2\,dm +a\int\bar{z}|f|^2|Qf|^2\,dm \Bigg\}.
\end{align*}

Next, we bound $|w|$ from below using the triangle inequalities for sums and integrals.   This yields
\begin{align*}
     |w|   &= \Bigg| 2a\int|f|^2\,dm -a\int Qf|f|^2\,dm - a\int \overline{Qf}|f|^2\,dm + a\int |f|^2|Qf|^2\,dm \\
          &\qquad   - aw\int\bar{z}|f|^2\,dm -\int \overline{zf}|Qf|^2\,dm  \Bigg|\\
          &\qquad \div\Bigg|  a\int\bar{z}|f|^2\,dm  -a\int\bar{z}Qf|f|^2\,dm  - a\int\overline{zQf}|f|^2\,dm +a\int\bar{z}|f|^2|Qf|^2\,dm \Bigg| \\
           &\geq  \Bigg\{ 2|a|\int|f|^2\,dm -\bigg[|a|\int |Qf||f|^2\,dm + |a|\int |Qf||f|^2\,dm + |a|\int |f|^2|Qf|^2\,dm \\
          &\qquad   + |aw|\int|f|^2\,dm +\int |f||Qf|^2\,dm \bigg] \Bigg\}\\
          &\qquad \div\Bigg\{  |a|\int|f|^2\,dm  +|a|\int|Qf||f|^2\,dm + |a|\int|Qf||f|^2\,dm +|a|\int|f|^2|Qf|^2\,dm \Bigg\} \\
\end{align*}

Apply the following estimates
\begin{align*}
     |Q(z)| &\leq  |a|(1 + |w|) \leq 2|a|  \\
     \int |f|^3\,dm &\leq \Big( \int|f|^2 \,dm  \Big)^{1/2}\Big( \int|f|^4 \,dm \Big)^{1/2} \leq \|f\|_2 \|f\|_4^2 = \|f\|_2,
\end{align*}
with occurrences of $\int |f|^2\,dm$ recorded as $\|f\|_2^2$.  The bounds continue with
\begin{align*}
     |w|    &\geq  \Bigg\{ 2|a|\|f\|_2^2 -\bigg[2|a|^2\|f\|_2  + 2|a|^2\|f\|_2 + 4|a|^3  
            + |aw|\|f\|_2^2 +4|a|^2\|f\|_2 \bigg] \Bigg\}\\
          &\qquad \div\Bigg\{  |a|\|f\|_2^2  +2|a|^2\|f\|_2 + 2|a|^2\|f\|_2 +4|a|^3 \Bigg\} \\
          &=  \Bigg\{ 2|a|\|f\|_2^2 -\bigg[8|a|^2\|f\|_2   + 4|a|^3  
            + |aw|\|f\|_2^2  \bigg] \Bigg\}\\
          &\qquad \div\Bigg\{  |a|\|f\|_2^2  +4|a|^2\|f\|_2  +4|a|^3 \Bigg\}.
\end{align*}

This implies
\begin{align}
      |aw|\|f\|_2^2  +4|a|^2|w|\|f\|_2  +4|a|^3|w| + 8|a|^2\|f\|_2   + 4|a|^3  
            + |aw|\|f\|_2^2 &\geq 2|a|\|f\|_2^2\nonumber \\
              2|a|^2(2+|w|)\|f\|_2   + 2|a|^3(1+|w|)  
             &\geq |a|(1 - |w|)\|f\|_2^2.  \label{firstineqfnormsub2sq}
\end{align}
By solving the last inequality we find that
\[
    \|f\|_2 \leq \frac{2|a|^2(2+|w|) + \sqrt{4|a|^4(2+ |w|)^2 + 8|a|^4(1- |w|^2)}}{2|a|(1-|w|)}.
\]
In particular,  this must be true in the limit as $|w|$ tends toward zero (through which the right hand side decreases), so we get
\begin{equation}\label{f2normvsa}
     \|f\|_2 \leq  C|a|,
\end{equation}
where we can take $C= (2 + \sqrt{6})$.

\bigskip

Let us set \eqref{f2normvsa} aside for now, and perform the analogous calculation starting with the formula
\[
       w= \frac{a\bar{D}-A}{aC},
\]
also from Theorem \ref{formulasforw}.

The result is
\begin{align*}
     w &=   \Bigg\{   a\int\bar{z}|f|^2\,dm  -a\int\bar{z}Qf|f|^2\,dm  - a\int\overline{zQf}|f|^2\,dm +a\int\bar{z}|f|^2|Qf|^2\,dm \\
         &\qquad          - \bigg[\overline{f(0)} - \int\bar{f}Qf\,dm -\int\overline{f^2Q}\,dm + \int \bar{f}|Qf|^2\,dm \bigg] \Bigg\} \\
         &\qquad  \div \Bigg\{   a\int|f|^2\,dm -a\int Qf|f|^2\,dm - a\int \overline{Qf}|f|^2\,dm + a\int |f|^2|Qf|^2\,dm  \Bigg\}\\
     |w| &=   \Bigg|   a\int\bar{z}|f|^2\,dm  -a\int\bar{z}Qf|f|^2\,dm  - a\int\overline{zQf}|f|^2\,dm +a\int\bar{z}|f|^2|Qf|^2\,dm \\
         &\qquad          - \bigg[\overline{f(0)} - \int\bar{f}Qf\,dm -\int\overline{f^2Q}\,dm + \int \bar{f}|Qf|^2\,dm \bigg] \Bigg| \\
         &\qquad  \div \Bigg|  a\int|f|^2\,dm -a\int Qf|f|^2\,dm - a\int \overline{Qf}|f|^2\,dm + a\int |f|^2|Qf|^2\,dm  \Bigg|  
\end{align*}
\begin{align*}
         &\geq   \Bigg\{ |{f(0)}| - \bigg[ |a|\int|f|^2\,dm  +|a|\int|Qf||f|^2\,dm + |a|\int|Qf||f|^2\,dm +|a|\int|f|^2|Qf|^2\,dm \\
         &\qquad         + \int|f||Qf|\,dm +\int|f^2Q|\,dm + \int|f||Qf|^2\,dm \bigg] \Bigg\} \\
         &\qquad  \div \Bigg\{  |a|\int |f|^2\,dm+|a|\int |Qf||f|^2\,dm + |a|\int |Qf||f|^2\,dm + |a|\int |f|^2|Qf|^2\,dm  \Bigg\}\\
         &\geq   \Bigg\{ |{f(0)}| - \bigg[5 |a|\|f\|_2^2  +8|a|^2\|f\|_2 +4|a|^3
            \bigg] \Bigg\} 
         \div \Bigg\{  |a|\|f\|_2^2+4|a|^2\|f\|_2 + 4|a|^3  \Bigg\}.
\end{align*}

This yields
\begin{equation}\label{ineqforfnought}
       |aw|\|f\|_2^2+4|a|^2|w|\|f\|_2 + 4|a|^3|w| + 5 |a|\|f\|_2^2  +8|a|^2\|f\|_2 +4|a|^3 \geq |f(0)|     
\end{equation} 

Now apply \eqref{f2normvsa}, to replace the left side with a larger bound:
\begin{align*}
      |aw|C^2|a|^2+4|a|^2|w|C|a| + 4|a|^3|w| + 5 |a|C^2|a|^2  +8|a|^2 C|a| +4|a|^3 \geq |f(0)|.
\end{align*}

The conclusion is that
\begin{equation}\label{fnoughtvsacubed}
     |f(0)| \leq C'|a|^3,
\end{equation}
where, with $|w| \leq 1$, we can take
\[
    C' =  6C^2 + 12C + 8.
\]

\bigskip

Finally, combine \eqref{fnoughtvsacubed} with Lemma \ref{redbellpepper2}, to the effect that
\[
     C'|a^3w| \geq |wf(0)| \geq C_1|a|^3.
\]
for some $C_1 > 0$.

But then, this implies that
\[
     |w| \geq \frac{C_1}{C'},
\]
contradicting the assumption that $|w|$ can be made arbitrarily small. 
\end{proof}

This proves Theorem \ref{rootsbddevenint} when $p=4$.
Our aim now is to extend this argument to the general case, in which $p$ is any even integer exceeding four. Before doing so, we need two lemmas. The key observation for this extension is contained in the following lemma.

\begin{Lemma}\label{keyobs}
   Let $2 < p < \infty$, and suppose that $2 \leq k < p$.  Then for any $f \in H^p$ with $\|f\|_p = 1$, we have
   \[
        \int |f|^k\,dm \leq \|\phi\|_{p-2}^{p-k},
   \]
   where $\phi = |f|^{2/(p-2)}$.
\end{Lemma}

\begin{proof}  For $k=2$ the claim follows from
\[
   \int|f|^2\,dm = \int\phi^{p-2}\,dm = \|\phi\|_{p-2}^{p-2}.
\]

Next, let $2<k<p$.  From
\[
       \frac{p-k}{p-2} + \frac{k-2}{p-2} = 1
\]     
we see that $(p-2)/(p-k)$ and $(p-2)/(k-2)$ are H\"{o}lder conjugates to each other.  

Notice also that
\[
     \bigg[  k - \frac{2(p-k)}{p-2} \bigg]\frac{p-2}{k-2} = \bigg[\frac{pk-2k-2p+2k}{p-2}\bigg]\frac{p-2}{k-2} = p.
\]

We now apply H\"{o}lder's inequality, resulting in
\begin{align*}     
      \int |f|^k\,dm  &=  \int |f|^{2(p-k)/(p-2)} |f|^{k- 2(p-k)/(p-2)} \, dm   \\
           &\leq  \Big( \int |f|^{[2(p-k)/(p-2)]\cdot [(p-2)/(p-k)]}  \, dm\Big)^{(p-k)/(p-2)}  \\
            &\qquad       \times \Big( \int  |f|^{[k- 2(p-k)/(p-2)]\cdot [ (p-2)/(k-2) ]} \, dm \Big)^{(k-2)/(p-2)}  \\
            &\leq \bigg\{ \Big( \int |f|^{[2/(p-2)]\cdot (p-2)}  \, dm\Big)^{1/(p-2)} \bigg\}^{p-k} 
            \Big( \int  |f|^{p} \, dm \Big)^{(k-2)/(p-2)}\\
            &\leq \bigg\{ \Big( \int |\phi|^{p-2}  \, dm\Big)^{1/(p-2)} \bigg\}^{p-k} 
            \bigg\{ \Big( \int  |f|^{p} \, dm \Big)^{1/p} \bigg\}^{p(k-2)/(p-2)}  \\
            &= \|\phi\|_{p-2}^{p-k}\|f\|_p^{p(k-2)/(p-2)}\\
            &= \|\phi\|_{p-2}^{p-k}.
\end{align*}    
\end{proof}

In addition to the lemma above, we'll also need the following elementary fact.
\begin{Lemma}  \label{elemfact}
Suppose that $n$ is an integer greater than 2.
   Let $\{x_k\}$ and $\{y_k\}$ be positive sequences, and suppose that for some $a_0 >0$ and nonnegative constants $a_1, a_2,\ldots, a_n$ we have
   \[
            a_0 x_k^n +a_1x_k ^{n-1}y_k + a_2 x_k^{n-2}y_k^2 + \cdots + a_{n-1}x_k y_k^{n-1} \geq y_k^n,\ \,\forall k.
   \]
   Then there exists a constant $C>0$ such that
   \[
       Cx_k \geq  y_k,\ \,\forall k.
   \]
\end{Lemma}

\begin{proof}
    First, by selecting possibly larger values of the constants $a_0, a_1, a_2,\ldots, a_n$, we can find $C'>0$ such that
  \[
        C' x_k(x_k^{n-1} + y_k^{n-1}) \geq y_k^{n}, \ \,\forall k.
  \]
  
  For those indices $k$ satisfying $x_k \geq y_k$, we have
  \[
         2C' x_k^n \geq C' x_k(x_k^{n-1} + y_k^{n-1}) \geq y_k^{n}.
  \]
  This implies
  \[
         (2C')^{1/n} x_k \geq y_k
  \]
  for such $k$.

  For the remaining indices $k$, in which $x_k < y_k$, we have
  \[
       2C' x_k y_k^{n-1} \geq  C' x_k(x_k^{n-1} + y_k^{n-1}) \geq y_k^{n}.
  \]
  This implies
  \[
         2C' x_k \geq y_k
  \]
  for such $k$.
  
Thus, the claim holds, with $C =   (2C')^{1/n} + 2C'$. 
\end{proof}

We are now equipped to complete the full proof of Theorem \ref{rootsbddevenint}.

\begin{proof}(Theorem \ref{rootsbddevenint}, $p>4$ an even integer).
     
We begin by following the calculations from the $p=4$ case, using estimates of the form
\[
       \int |f|^j |Qf|^k\,dm  \leq |2a|^k\|\phi\|_{p-2}^{p-j-k},
\]
arising from Lemma \ref{keyobs}.

During these calculations, the quantity $\|\phi\|_{p-2}$ takes over the role of $\|f\|_2$ from the $p=4$ case. (In fact, when $p=4$, we have 
\[
        \|\phi\|_{p-2} = \Big(\int |\phi|^{p-2}\,dm\Big)^{1/(p-2)} = \Big( \int |f|^2 \,dm\Big)^{1/2} = \|f\|_2.)
\]

The outcome analogous to \eqref{firstineqfnormsub2sq} is that
\[
       b_0 |a|^{p-1} + b_1 |a|^{p-2} \|\phi\|_{p-2}+ b_2 |a|^{p-3} \|\phi\|_{p-2}^2 + \cdots + b_{p-2} |a| \|\phi\|_{p-2}^{p-2} \geq \|\phi\|_{p-2}^{p-1},
\]
for some positive constants $b_0$, $b_1$,\ldots, $b_{p-2}$.

As a result of Lemma \ref{elemfact}, there is a positive constant $C$ such that
\[
         \|\phi\|_{p-2} \leq C|a|.
\]

Analogous to \eqref{ineqforfnought}, we have
\[
      B_1(1 + |w|) |a|  \|\phi\|_{p-2}^{p-2} + B_2(1 + |w|) |a|^2  \|\phi\|_{p-2}^{p-3} + \cdots + B_{p-1} (1 + |w|)|a|^{p-1} \geq |f(0)|, 
\]
for some positive constants $B_1$, $B_2$,\ldots, $B_{p-1}$.   As a consequence,
\[
       |f(0)| \leq C'|a|^{p-1}
\]
for some $C'>0$.

Once again, we invoke Lemma \ref{redbellpepper2}.  The end result is 
\[
       C' |wa^{p-1}| \geq  |wf(0)| \geq C_1|a|^{p-1},
\]     
which implies
\[
        |w| \geq \frac{C_1}{C'},
\]
 contradicting the assumption that $w$ could be made arbitrarily close to zero.  This completes the proof of Theorem \ref{rootsbddevenint}.
\end{proof}
\bigskip

 We end by noting again that Theorems \ref{rootsbddinnerf} and \ref{rootsbddevenint} show, in two special cases, the roots of a nontrivial OPA must lie outside a disk of some radius $\rho>0$, centered at the origin.  It would desirable to extend this result to all $f \in H^p$, and to all $p$, $1<p<\infty$.  It would also be particularly interesting to determine whether $\rho = 1$.  This situation invites further investigation.


\bibliographystyle{plain}

\bibliography{biblio.bib}


\end{document}